\documentclass{birkjour}
 \newtheorem{thm}{Theorem}[section]
 \newtheorem{cor}[thm]{Corollary}
 \newtheorem{lem}[thm]{Lemma}
 
 \theoremstyle{definition}
 
 \theoremstyle{remark}
 \newtheorem{rem}[thm]{Remark}
 
 \numberwithin{equation}{section}

\newcommand{\CC}{{\mathbb C}}
\renewcommand{\i}{\mathrm{i}}
\newcommand{\bI}{{\bf I}}
\newcommand{\bM}{{\bf M}}
\newcommand{\bN}{{\bf N}}
\renewcommand{\Re}{\mathop{\mathrm{Re}}}
\renewcommand{\Im}{\mathop{\mathrm{Im}}}
\newcommand{\codim}{{\rm codim}}

\newcommand{ \Null}{{\rm Null}}

\begin{document}

%
%
%
%
%
%
%
%
%

\title[General conjugation problem]
 {A boundary integral method for the general conjugation problem in multiply connected circle domains}

\author[M Nasser]{Mohamed M.S. Nasser}

\address{%
Department of Mathematics, Statistics and Physics\\
Qatar University, P.O.Box 2713, Doha\\
Qatar}

\email{mms.nasser@qu.edu.qa}

\thanks{This paper has been presented in: BFA 3rd meeting, Rzeszow, Poland, April 20--23, 2016. The author is grateful to Qatar University for the financial support to attend the meeting and to professor Piotr Drygas, chairman of the organizing committee of the meeting, for the hospitality during the meeting.}


\date{\today}

\begin{abstract}
We present a boundary integral method for solving a certain class of Riemann-Hilbert problems known as the general conjugation problem. The method is based on a uniquely solvable boundary integral equation with the generalized Neumann kernel. We present also an alternative proof for the existence and uniqueness of the solution of the general conjugation problem.
\end{abstract}

\maketitle

\section{Introduction}

The Riemann-Hilbert problem (RH-problem, for short) is one of the most important classes of boundary value problems for analytic functions. Indeed, the Dirichlet problem, the Neumann problem, the mixed Dirichlet-Neumann problem, the problems of computing the conformal mapping and the external potential flow can be formulated as RH-problems. 
The RH problem consists of determining of all analytic functions in a domain $G$ in the extended complex plane that satisfy a prescribed boundary condition on the boundary $C=\partial G$.

A boundary integral equation with continuous kernel for solving the RH problem has been derived in~\cite{MurRazNas02,MurNas03}. The Kernel of the derived integral equation is a generalization of the well known \emph{Neumann kernel}. So, the new kernel has been called the \emph{generalized Neumann kernel}. The solvability of the boundary integral equation with the generalized Neumann kernel has been studied for simply connected domains with smooth boundaries in~\cite{WegMurNas05}, for simply connected domains with piecewise smooth boundaries in~\cite{Nas-cvee08}, and for multiply connected domains in~\cite{WegNas08,NasIbb09}. 

It turns out that the solvability of the boundary integral equation with the generalized Neumann kernel depends on the index $\kappa_j$ of the coefficient function $A$ of the RH-problem on each boundary component of the boundary of $G$. However, the solvability of the RH-problem depends on the total index $\kappa$ of the function $A$ on the whole boundary of $G$. This raises a difficulty in using the boundary integral equation with the generalized Neumann kernel to solve the RH-problem in multiply connected domains. Such a difficulty does not appear for the simply connected case since the boundary of $G$ consists of only one component. So far, the boundary integral equation with the generalized Neumann kernel has been used to solve the RH-problem in multiply connected domains for special case of the function $A$ (see~\cite[Eq.~(1.2)]{Nas-etna15}). For such special case, the boundary integral equation with the generalized Neumann kernel has been used successfully to compute the conformal mapping onto more than $40$ canonical domains~\cite{Nas-cmft09,Nas-siam09,Nas-jmaa11,Nas-jmaa13,Nas-cmft15,Nas-siam13,Nas-cmft16} and to solve several boundary value problems such as the Dirichlet problem, the Neumann problem, and the mixed boundary value problem~\cite{Alh13,Nas-jam12,Nas-amc11}.

In this paper, we shall use the boundary integral equation with the generalized Neumann kernel to solve a certain class of RH-problems in multiply connected domains considered by Wegmann~\cite{Weg01r} and known as the \emph{general conjugation problem}. We shall also use the integral equation to provide an alternative proof for the existence and uniqueness of the solution of the general conjugation problem.

\section{The generalized Neumann kernel}

Let $G$ be the unbounded multiply connected circular domain obtained by removing $m$ disks $D_1, \ldots, D_m$ from the extended complex plane $\CC\cup\{\infty\}$ such that $\infty\in G$ (see Figure~\ref{f:dom}). The disk $D_j$ is bounded by the circle $C_j=\partial D_j$ with a center $z_j$ and radius $r_j$. We assume that each circle $C_j$ is clockwise oriented and parametrized by 
\[
\eta_j(t)=z_j+r_j e^{-\i t}, \quad t\in J_j:=[0,2\pi], \quad j=1,2,\ldots,m.
\]
Let $J$ be the disjoint union of the $m$ intervals $J_1,\ldots,J_m$ which is defined by 
\begin{equation}\label{e:J}
J = \bigsqcup_{j=1}^{m} J_j=\bigcup_{j=1}^{m}\{(t,j)\;:\;t\in J_j\}.
\end{equation}
The elements of $J$ are order pairs $(t,j)$ where $j$ is an auxiliary index indicating which of the intervals the point $t$ lies in. Thus, the parametrization of the whole boundary $C=\partial D=C_1\cup C_2\cup\cdots\cup C_m$ is defined as the complex function $\eta$ defined on $J$ by
\begin{equation}\label{e:eta-1}
\eta(t,j)=\eta_j(t), \quad t\in J_j,\quad j=1,2,\ldots,m.
\end{equation}
We assume that for a given $t$ that the auxiliary index $j$ is known, so we replace the pair $(t,j)$ in the left-hand side of~(\ref{e:eta-1}) by $t$, i.e., for a given point $t\in J$, we always know the interval $J_j$ that contains $t$. The function $\eta$ in~(\ref{e:eta-1}) is thus simply written as
\begin{equation}\label{e:eta}
\eta(t):= \left\{ \begin{array}{l@{\hspace{0.5cm}}l}
\eta_1(t),&t\in J_1=[0,2\pi],\\
\hspace{0.3cm}\vdots\\
\eta_m(t),&t\in J_m=[0,2\pi].
\end{array}
\right .
\end{equation}

\begin{figure}
\centerline{\scalebox{0.5}[0.5]{
\includegraphics{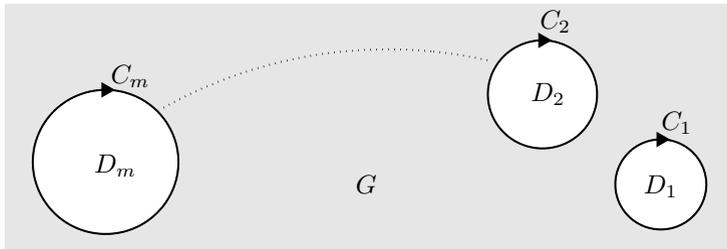}}}
 \vskip-3.4cm \noindent\hspace{+4.8cm} $C_2$
 \vskip+.30cm \noindent\hspace{-6.4cm} $C_m$
 \vskip+.20cm \noindent\hspace{+8.0cm} $C_1$
 \vskip-0.8cm \noindent\hspace{+4.6cm} $D_2$
 \vskip+.80cm \noindent\hspace{+3.9cm}$G$ \hspace{3.4cm}$D_1$
 \vskip-0.7cm \noindent\hspace{-6.8cm} $D_m$
 \vskip+0.8cm
 \caption{\rm An unbounded multiply connected circular domain $G$ of connectivity $m$.}
 \label{f:dom}
\end{figure}

Let $H$ denote the space of all real functions $\gamma$ in J, whose restriction $\gamma_j$ to $J_j=[0,2\pi]$ is a
real-valued, $2\pi$-periodic and H\"older continuous function for each $j=1,\ldots,m$, i.e.,
\[
\gamma(t) = \left\{
\begin{array}{l@{\hspace{0.5cm}}l}
 \gamma_1(t),     & t\in J_1, \\
  \vdots,       & \\
 \gamma_m(t),     & t\in J_m. \\
\end{array}%
\right.
\]
In view of the smoothness of the parametrization $\eta$, a real H\"older continuous function $\hat\gamma$ on $C$ can be interpreted via $\gamma(t):=\hat\gamma(\eta(t))$, $t\in J$, as a function $\gamma\in H$; and vice versa. So, in this paper, for any given complex or real valued function $\phi$ defined on $C$, we shall not distinguish between $\phi(t)$ and $\phi(\eta(t))$. Further, for any complex or real valued function $\phi$ defined on $C$, we shall denote the restriction of the function $\phi$ to the boundary $C_j$ by $\phi_j$, i.e., $\phi_j(\eta(t))=\phi(\eta_j(t))$ for each $j=1,\ldots,m$. However, if $\Phi$ is an analytic function in the domain $G$, we shall denote the restriction of its values to the boundary $C_j$ by $\Phi_{|j}$, i.e., $\Phi_{|j}(\eta(t))=\Phi(\eta_j(t))$ for each $j=1,\ldots,m$.

Let $A$ be a continuously differentiable complex function on $C$ with $A\ne0$. The generalized Neumann kernel is defined by (see~\cite{WegNas08} for details)
\begin{equation}\label{e:N} 
N(s,t):=  \frac{1}{\pi}\Im\left(
\frac{A(s)}{A(t)}\frac{\dot\eta(t)}{\eta(t)-\eta(s)}\right).
\end{equation}
When $A=1$, the kernel $N$ is the well-known Neumann kernel which appears frequently in the integral equations of potential theory and conformal mapping (see e.g.~\cite{Hen86}). We define also the following singular kernel $M(s,t)$ which is closely related to the generalized Neumann kernel $N(s,t)$~\cite{WegNas08},
\begin{equation}\label{e:M} 
M(s,t):=  \frac{1}{\pi}\Re\left(
\frac{A(s)}{A(t)}\frac{\dot\eta(t)}{\eta(t)-\eta(s)}\right).
\end{equation}

\begin{lem}[\cite{WegNas08}]
{\rm(a)} The kernel $N(s,t)$ is continuous with
\begin{equation}
\label{e:N-d} N(t,t)= \frac{1}{\pi} \left(\frac{1}{2}\Im
\frac{\ddot\eta(t)}{\dot \eta(t)} -\Im\frac{\dot A(t)}{ A(t)}
\right).
\end{equation}
{\rm(b)} When  $s,t\in J_j$  are in the same parameter interval $J_j$, then
\begin{equation}\label{e:M-tM}
M(s,t)= -\frac{1}{2\pi} \cot \frac{s-t}{2} + M_1(s,t)
\end{equation}
with a continuous kernel $M_1$ which takes on the diagonal the
values
\begin{equation}\label{e:M-tM-d}
M_1(t,t)= \frac{1}{\pi} \left(\frac{1}{2}\Re
\frac{\ddot\eta(t)}{\dot \eta(t)} -\Re \frac{\dot A(t)}{ A(t)}
\right).
\end{equation}
\end{lem}

On $H$ we define the Fredholm operator
\[
\bN\gamma=\int_JN(s,t)\gamma(t)dt
\]
and the singular operator
\[
\bM\gamma=\int_JM(s,t)\gamma(t)dt.
\]
Both operators $\bN$ and $\bM$ are bounded on the space $H$ and map $H$ into itself. For more details, see~\cite{WegNas08}. The identity operator on $H$ is denoted by $\bI$.

\section{The Riemann-Hilbert problem}

The RH-problem for the unbounded multiply connected domain $G$ is defined as follows:\\
For a given function $\gamma\in H$, search a function $\Psi$ analytic in $G$ and continuous on the closure
$\overline G$ with $\Psi(\infty)=0$ such that  the boundary values of $\Psi$ satisfy on $C$ the boundary condition
\begin{equation}\label{eq:RHp}
\Re [A \Psi]=\gamma.
\end{equation}

The boundary condition in~(\ref{eq:RHp}) is non-homogeneous. When $\gamma\equiv0$, we have the homogeneous boundary condition 
\begin{equation}\label{eq:RHp-h}
\Re [A \Psi]=0.
\end{equation}

The solvability of the RH problem depends upon the {\it index} of the function $A$ on the boundary $C$. The index $\kappa_j$ of the function $A$ on the circle $C_j$ is the change of the argument of $ A$ along the circle $C_j$ divided by $2\pi$. The index $\kappa$ of the function $ A$ on the whole boundary curve $C$ is the sum 
\begin{equation}\label{eq:ind}
\kappa=\sum_{j=1}^m  \kappa_j.
\end{equation}

\begin{rem}
Vekua~\cite[Eq.~(1.2), p. 222]{Vek92}, Gakhov~\cite[Eq.~(27.1), p. 208]{Gak66} and Mityushev~\cite[Eq.~(38.3), p. 601]{Mit11} define the RH-problem with $\Re[\overline{A}\Psi]=\gamma$, i.e., with the complex conjugate of the function $A$. This has the consequence that in some of the later results the index of the function $A$ occurs with the opposite sign
as in~\cite{Vek92,Gak66,Mit11}.
\end{rem}

We follow~\cite{WegNas08} and define the space $R^+$, the spaces of functions $\gamma$ for which the RH problem~\eqref{eq:RHp} have solution, by
\begin{equation}\label{eq:R+} 
R^+ := \{ \gamma \in H: \gamma =\Re [A\Psi]\, \mbox{ on $C$},\, \Psi \mbox{ analytic in $G$, $\Psi(\infty)=0$}\}.
\end{equation}
We define also the space $S^+$ to be the space of the boundary values of solutions of the homogeneous RH problem, i.e.,
\begin{equation}\label{e:S+} 
S^+ := \{ \gamma \in H :\gamma=A\Psi\, \mbox{ on $C$},\, \Psi \mbox{ analytic in $G$, $\Psi(\infty)=0$}\}.
\end{equation}

To study the solvability of the RH-problem~\eqref{eq:RHp}, we define the following boundary value problem as the homogeneous exterior RH problem on $G^-=D_1\cup D_2\cup\cdots\cup D_m$:\\
Search a function $g$ analytic in $G^-$ and continuous on the closure $\overline {G^-}$ such that the boundary values of $g$ satisfy on $C$,
\begin{equation}\label{eq:eRHp} 
\Re [Ag]=0.
\end{equation}

It is clear that $G^-$ is not a domain. In fact, it is the union of $m$ disjoint disks. So, solving the homogeneous exterior RH problem~(\ref{eq:eRHp}) is equivalent to solving $m$ RH problems in the disks $D_1, D_2, \ldots, D_m$. The space of the boundary values of solutions of the homogeneous exterior RH problem~(\ref{eq:eRHp}) is denoted by $S^-$, i.e.,
\begin{equation}\label{eq:S-} 
S^- := \{ \gamma \in H:\gamma=Ag\, \mbox{ on $C$},\, g \mbox{ analytic in $G^-$}\}.
\end{equation}
For $j=1,2,\ldots,m$, let $S^-_j$ be the subspace of $S^−$ of real functions $\gamma\in S^-$ such that
\[
\gamma(t)= \left\{ \begin{array}{l@{\hspace{0.5cm}}l}
0,&t\in J_k, \; k\ne j, \; k=1,2,\ldots,m,\\
A_j(t)g(\eta_j(t)),&t\in J_j,
\end{array}
\right .
\]
where $g$ is analytic in the disk $D_j$, i.e., $g$ is a solution of the following homogenous RH-problem on the disk $D_j$,
\begin{equation}\label{eq:g-Dj}
\Im[A_jg]=0 \quad {\rm on}\quad C_j.
\end{equation}
The problem~\eqref{eq:g-Dj} is a RH-problem in the bounded simply connected domain $D_j$ and the index of the function $A_j$ on $C_j=\partial D_j$ is $\kappa_j$. Then we have from~\cite[Eq.~(29)]{WegMurNas05}
\begin{equation}\label{eq:sim-S-j}
\dim(S_j^-)=2\kappa_j+1.
\end{equation}
(note that the orientation of the circles $C_j$ is clockwise which changes the sign in ~\cite[Eq.~(29)]{WegMurNas05}).
Then, we have the following lemmas from~\cite{WegNas08}.
\begin{lem}\label{L:S-Sk}
The space $S^-$ is the direct sum of the subspaces $S_1^-,S_2^-,\ldots,S_m^-$,
\begin{equation}\label{eq:S-Sk}
S^-=S_1^-\oplus S_2^-\oplus \cdots \oplus S_m^-.
\end{equation}
\end{lem}

The space $S^-$ plays a very important rule in using the boundary integral equation with the generalize Neumann kernel to solve the RH-problem~\eqref{eq:RHp} especially when the problem is not solvable since it allows us to find the form of conditions that we should impose on $\gamma$ to make the problem solvable. For simply connected domains, it was proved in~\cite[Corollary~3]{WegMurNas05} that the space $H$ has direct sum decomposition
\begin{equation}\label{eq:H=R+S}
H=R^+\oplus S^-.
\end{equation}
The decomposition~\eqref{eq:H=R+S} means that if the RH-problem $\Re[A\Psi]=\gamma$ is not solvable, then there exists a unique function $h\in S^-$ such that the RH-problem $\Re[A\Psi]=\gamma+h$ is solvable. For simply connected domains, the decomposition~\eqref{eq:H=R+S} is valid for general index $\kappa$ of the function $A$. 
However, for multiply connected domains, the decomposition~\eqref{eq:H=R+S} in general is not correct since we may have $R^+\cap S^-\ne\{0\}$ (see \cite[\S10]{WegNas08}.) 

In this paper, we shall consider special case of the function $A$ for which we can prove that the decomposition~\eqref{eq:H=R+S} is valid for multiply connected domains (see Theorem~\ref{T:H=R+S} below), namely, we assume the index of the function $A$ satisfies
\begin{equation}\label{eq:kappa-j}
\kappa_j\ge0\quad{\rm for all}\quad j=1,2,\ldots,m, 
\end{equation}
which implies that $\kappa\ge0$. RH-problem with such special case of the index has wide applications. For example, our assumption~\eqref{eq:kappa-j} on the index are satisfied for the RH-problem used in~\cite{Nas-cmft09,Nas-siam09,Nas-jmaa11,Nas-jmaa13,Nas-cmft15,Nas-siam13,Nas-cmft16} to develop a method for computing the conformal mapping onto more than $40$ canonical domains and for the RH-problems studied in~\cite{Mit11,Mit12,Weg01r}. Furthermore, many other boundary value problems such as the Dirichlet problem, the Neumann problem, the mixed boundary value problem, and the Schwarz problem can be reduced to RH-problems whose indexes satisfy our assumption~\eqref{eq:kappa-j} (see e.g.,~\cite{Alh13,Nas-jam12,Nas-amc11}).

Under the above assumption~\eqref{eq:kappa-j} on the index, we have the following theorem from~\cite{WegNas08,NasIbb09}.

\begin{thm}\label{T:RHp-sol}
For $\kappa\ge0$, we have
\begin{equation}\label{eq:dim-RS+}
\codim(R^+)=2\kappa+m, \quad \dim(S^+)=0.
\end{equation}
\end{thm} 

The following lemma follows from~(\ref{eq:sim-S-j}),~(\ref{eq:S-Sk}) and~\eqref{eq:ind}.
\begin{lem}\label{L:RHp-sol2}
Let $\kappa_j\ge0$ for $j=1,2,\ldots,m$, then
\begin{equation}\label{eq:dim-S-}
\dim(S^-)=2\kappa+m.
\end{equation}
\end{lem} 

There is a close connection between RH problems and integral equations with the generalized Neumann kernel. The null-spaces of the operators $\bI\pm\bN$ are related to the spaces $S^\pm$ by (see~\cite[Theorem~11, Lemma~20]{WegNas08}
\begin{eqnarray}
\label{e:S-=I+N}
\Null(\bI+\bN)&=&S^-, \\ 
\label{e:S+=I-N}
\Null(\bI-\bN)&=&S^+\oplus W,
\end{eqnarray}
where $W$ is isomorphic via $\bM$ to $R^+\cap S^-$. For the function $A$ defined by~\eqref{eq:A}, we have~\cite{WegNas08,NasIbb09}
\begin{equation}\label{eq:Null-I-N}
\dim(\Null(\bI+\bN))=\dim(S^-)=\sum_{j=1}^{m}\max(0,2\kappa_j+1)=2\kappa+m
\end{equation}
and 
\begin{equation}\label{eq:Null-I+N}
\dim(\Null(\bI-\bN))=\sum_{j=1}^{m}\max(0,-2\kappa_j-1)=0.
\end{equation}
In view of~\eqref{e:S+=I-N}, Eq.~\eqref{eq:Null-I+N} implies that $S^+=W=\{0\}$ (see also \eqref{eq:dim-RS+}). Since $W$ isomorphic to $R^+\cap S^-$, we have also
\begin{equation}\label{eq:R-S-}
R^+\cap S^-=\{0\}.
\end{equation}

\begin{thm}\label{T:H=R+S}
Let $\kappa_j\ge0$ for $j=1,2,\ldots,m$, then the space $H$ has the decomposition
\begin{equation}\label{eq:H=R+S2}
H=R^+\oplus S^-
\end{equation}
\end{thm}
\begin{proof}
Since 
\[
\codim(R^+)=\dim(S^-)=2\kappa+m, 
\]
 the direct sum decomposition follows from~(\ref{eq:R-S-}).
\end{proof}

It follows from Theorem~\ref{T:RHp-sol} that the non-homogeneous RH problem~(\ref{eq:RHp}) is in general insolvable for $\kappa_j\ge0$. If it is solvable, then the solution is unique. The following corollary which follows from Theorem~\ref{T:H=R+S} provides us with a way for modifying the right-hand side of~(\ref{eq:RHp}) to ensure the solvability of the problem.

\begin{cor}\label{C:g+h}
Let $\kappa_j\ge0$ for $j=1,2,\ldots,m$, then for any $\gamma\in H$, there exists a unique function $h\in S^-$ such that the following RH problem 
\begin{equation}\label{eq:RHp2}
\Re [A \Psi]=\gamma+h
\end{equation}
is uniquely solvable.
\end{cor}

Solving the RH-problem requires determining both the analytic function $\Psi$ as well as the real function $h$. This can be done easily using the boundary integral equations with the generalized Neumann kernel as in the following theorem.

\begin{thm}\label{T:ie-rhp}
Let $\kappa_j\ge0$ for $j=1,2,\ldots,m$. For any given $\gamma\in H$, let $\mu$ be the unique solution of the integral equation
\begin{equation}
\label{eq:ie} \mu - \bN \mu =-\bM \gamma.
\end{equation}
Then the boundary values of the unique solution of the RH problem~(\ref{eq:RHp2}) is given by
\begin{equation}\label{eq:Psi}
A\Psi=\gamma+h+\i\mu
\end{equation}
and the function $h$ is given by
\begin{equation}\label{eq:h}
h = [\bM\mu-(\bI-\bN)\gamma]/2.
\end{equation}
\end{thm}
\begin{proof}
The theorem can be proved using the same argument as in the proof of~\cite[Theorem~2]{Nas-cmft09}.
\end{proof}

The uniqueness of the solution of the integral equation~\eqref{eq:ie} follows from the Fredholm alternative theorem since $\dim(\Null(\bI-\bN))=0$. The advantages of Theorem~\ref{T:ie-rhp} are that it, based on the integral equation~\eqref{eq:ie}, provides us with formulas for computing the real function $h$ necessary for the solvability of the RH problem as well as the solution $\Psi$ of the RH problem.

\section{The general conjugation problem}

In this section, we shall consider a very important certain class of RH problems which has been considered by Wegmann~\cite{Weg01r}. The indexes of this class of RH problems satisfied our assumption~\eqref{eq:kappa-j}. This class has been considered by many researchers and has many applications~\cite{BalDel16,MitRog00,Weg01,Weg01r,Weg05}.

Wegmann~\cite{Weg01r} proves the following theorem.

\begin{thm}\label{T:Weg}
For any integer $\ell\ge0$ and for any sufficiently smooth functions $\gamma$ on the boundary of $G$, the RH problem
\begin{equation}\label{eq:RHp-w}
\Re\left[e^{\i\lambda_j}e^{\i\ell t}\Psi_{|j}(\eta(t))+(a_{j\ell}+\i b_{j\ell})e^{\i\ell t}+\cdots+(a_{j1}+\i b_{j1})e^{\i t}+a_{j0}\right]=\gamma_{j}(t), 
\end{equation}
has a unique solution consisting of an analytic function $\Psi$ in $G$ with $\Psi(\infty)=0$ and $(2\ell+1)m$ real numbers $a_{j0},a_{j1},\ldots,a_{j\ell},b_{j1},\ldots,b_{j\ell}$ for $j=1,\ldots,m$.
\end{thm}

Wegmann~\cite{Weg01r} called the RH problem~(\ref{eq:RHp-w}) as the \emph{general conjugation problem} since the case $\ell=0$ describes the problem of finding  the conjugate harmonic function of a harmonic function with boundary values $\gamma_j$~\cite{Weg01r}. The general conjugation problem~\eqref{eq:RHp-w} has been solved by Wegmann~\cite{Weg01r} using the method of \emph{successive conjugation} which reduces the problem~(\ref{eq:RHp-w}) to a sequence of RH problems on the circles $C_j$. This method has been first applied by Halsey~\cite{Hal79} for $\ell=0$. Applications of this problem to compute conformal mapping have been given in~\cite{Weg01,Weg05} for $\ell=1$ and in~\cite{BalDel16,Weg05} for $\ell=0$. 

In this paper, we shall present a method for solving the general conjugation problem~\eqref{eq:RHp-w} for any $\ell\ge0$. The method is based on the boundary integral equation with the generalized Neumann kernel~(\ref{eq:ie}). However, we shall first rewrite the problem in a form suitable for using the integral equation.

The boundary condition~\eqref{eq:RHp-w} can be written as
\begin{equation}\label{eq:RHp-w2}
\Re\left[e^{\i\lambda_k}e^{\i\ell t}\Psi_{|j}\right]=\gamma_{j}
-a_{j0}-\sum_{k=1}^{\ell}\Re\left[(a_{jk}+\i b_{jk})e^{\i kt}\right], \quad j=1,2,\ldots,m.
\end{equation}
Let $A$ be defined by
\begin{equation}\label{eq:A}
A(t):= \left\{ \begin{array}{l@{\hspace{0.5cm}}l}
e^{\i\lambda_1}e^{\i\ell t},&t\in J_1=[0,2\pi],\\
\hspace{0.3cm}\vdots\\
e^{\i\lambda_m}e^{\i\ell t},&t\in J_m=[0,2\pi].
\end{array}
\right .
\end{equation}
Let also $\psi$ be a function defined on the boundary $C$ where its values on $C_j$ is given by
\begin{equation}\label{eq:psi}
\psi_{j}(t)=-a_{j0}-\sum_{k=1}^{\ell}\Re\left[(a_{jk}+\i b_{jk})e^{\i kt}\right], \quad j=1,2,\ldots,m.
\end{equation}
Hence, the general conjugation problem~\eqref{eq:RHp-w} can be written as the RH-problem
\begin{equation}\label{eq:RHp-2}
\Re\left[A\Psi\right]=\gamma+\psi.
\end{equation}
By the existence and uniqueness of the solution general conjugation problem~\eqref{eq:RHp-w}, the RH-problem~\eqref{eq:RHp-2} has a unique solution. Remember that solving the RH-problem~\eqref{eq:RHp-2} requires determining the analytic function $\Psi$ and the unknown real function $\psi$. Determining the function $\psi$ is equivalent to determining the $m(2\ell+1)$ real constants $a_{j0}$, $a_{jk}$, and $b_{jk}$ in~\eqref{eq:RHp-w} for $j=1,\ldots,m$, $k=1,\ldots,\ell$.

The index of the function $A$ given by~\eqref{eq:A} is
\[
\kappa_j=\ell\ge0, \quad j=1,\ldots,m,
\]
and hence the total index is $\kappa=m\ell\ge0$. Thus our assumption~\eqref{eq:kappa-j} is satisfied.
We shall use the integral equation with the generalized Neumann kernel~(\ref{eq:ie}) to determine the analytic function $\Psi$ as well as the real function $\psi$ in~\eqref{eq:RHp-2}. However, we need to show first that the function $\psi$ is indeed in $S^-$. This can be proved by finding the explicit form of the functions of the space $S^-$ which will be given in the next section.

\section{The space $S^-$}

To find the explicit form of the space $S^-$, we need the following theorem from~\cite[\S~29.3]{Gak66}.

\begin{thm}\label{T:Gak}
Let $D=\{z:|z|<1\}$ and the boundary $C=\partial D$ is the unit circle parametrized by $\zeta(t)=e^{-\i t}$, $t\in[0,2\pi]$. For $\ell\ge0$, the solution of the following homogenous RH problem on $D$,
\[
\Re\left[\zeta(t)^{-\ell}g(\zeta(t))\right]=0, \quad \zeta(t)\in C,
\]
is given for $z\in\overline{D}$ by
\[
g(z)=z^\ell\left[\i c_0+\frac{1}{2}\sum_{k=1}^{\ell}\left(c_k z^k-\overline{c_k}z^{-k}\right)\right]
\]
where $c_0$ is an arbitrary real constant and $c_1,\ldots,c_\ell$ are arbitrary complex constants.
\end{thm}

\begin{lem}\label{L:Sk}
Let $h\in H$. Then $h\in S_j^-$ if and only if it has the form
\begin{equation}\label{eq:Sk-gam}
h(t)=\left\{
\begin{array}{ll}
	0, &\mbox{\rm on $C_k$ for $k\ne j$,}\\
	\beta_{j0}+\sum_{k=1}^{\ell}\Re\left[(\beta_{jk}+\i\alpha_{jk})e^{\i k t}\right], &\mbox{\rm on $C_j$},
\end{array}
\right.
\end{equation}
where $\beta_{j0},\beta_{jk},\alpha_{jk}$, $j=1,2,\ldots,m$, $k=1,2,\ldots,\ell$, are real constants.
\end{lem}
\begin{proof}
By the definition of the space $S_j^-$, a function $h\in S_j^-$ if and only if 
\begin{equation}\label{eq:Sk-gam2}
h(t)=\left\{\begin{array}{ll}
0,  & \mbox{ on $C_k$ for $k\ne j$},\\
A_j(t)g(\eta_j(t)), &	\mbox{ on $C_j$},
\end{array}\right.
\end{equation}
where $g$ is analytic in $D_j$ and $A_j(t)$ is the restriction of the function $A$ given by~\eqref{eq:A} to $J_j$.  Using the definition~\eqref{eq:A} of the function $A$, the function $\i g$ is a solution of the homogeneous RH problem
\begin{equation}\label{eq:g-2} 
\Re[e^{\i\lambda_j}e^{\i\ell t}(\i g(\eta_j(t)))]=0
\end{equation}
on the disk $D_j$. 
We define an analytic function $F$ in $D_j$ by
\[
F(z)=\i e^{\i\lambda_j} g(z).
\]
Then $F$ is a solution of the homogeneous RH problem 
\begin{equation}\label{eq:F-1} 
\Re[e^{\i\ell t} F(\eta_j(t))]=0
\end{equation}
on the disk $D_j$. The function
\[
\omega(z)=\frac{z-z_j}{r_j}
\]
is analytic on $D_j$ and maps the circle $C_j$ onto the unit circle and
\[
\omega(\eta_j(t))=\frac{\eta_j(t)-z_j}{r_j}=e^{-\i t}=\zeta(t)
\]
is the parametrization of the unit circle. Let the function $\hat F$ be defined on the unit disk $D$ by
\[
\hat F(z)=F(\omega^{-1}(z))=F(z_j+r_jz).
\]
Then, it follows from~\eqref{eq:F-1} that $\hat F$ is a solution of the homogeneous RH problem 
\begin{equation}\label{eq:hF-1} 
\Re\left[\zeta(t)^{-\ell} \hat F(\zeta(t))\right]=0
\end{equation}
in the unit disk $D$. Hence Theorem~\ref{T:Gak} implies that
\[
\zeta(t)^{-\ell} \hat F(\zeta(t))=\i\beta_{j0}+\frac{1}{2}\sum_{k=1}^{\ell}[(\alpha_{jk}+\i\beta_{jk})\zeta(t)^k-(\alpha_{jk}-\i\beta_{jk})\zeta(t)^{-k}]
\]
where $\beta_{j0},\beta_{jk},\alpha_{jk}$, $j=1,2,\ldots,m$, $k=1,2,\ldots,\ell$, are arbitrary real constants.
Since $\zeta(t)=e^{-\i t}$ and $\hat F(\zeta(t))=F(\eta_j(t))=\i e^{\i\lambda_j} g(\eta_j(t))$, we obtain
\begin{equation}\label{eq:F-sum1}
\i e^{\i\ell t} e^{\i\lambda_j} g(\eta_j(t))=\i\beta_{j0}+\frac{1}{2}\sum_{k=1}^{\ell}[(\alpha_{jk}+\i\beta_{jk})e^{-\i k t}-(\alpha_{jk}-\i\beta_{jk})e^{\i k t}].
\end{equation}
Since $h_{j}(t)=e^{\i\ell t} e^{\i\lambda_j} g(\eta_j(t))$, \eqref{eq:F-sum1} implies that
\[
h_{j}(t)=\beta_{j0}+\frac{1}{2}\sum_{k=1}^{\ell}\left[(\beta_{jk}-\i\alpha_{jk})e^{-\i k t}+(\beta_{jk}+\i\alpha_{jk})e^{\i k t}\right]
\]
which can be written as
\begin{equation}\label{eq:Sk-gam3}
h_{j}(t)=\beta_{j0}+\sum_{k=1}^{\ell}\Re\left[(\beta_{jk}+\i\alpha_{jk})e^{\i k t}\right].
\end{equation}
Hence~\eqref{eq:Sk-gam} follows from~\eqref{eq:Sk-gam2} and~\eqref{eq:Sk-gam3}.
\end{proof}

\begin{thm}\label{T:S-gam}
Let $h\in H$. Then $h\in S^-$ if and only if it has the form
\begin{equation}\label{eq:S-gam}
h_{j}(t)=\beta_{j0}+\sum_{k=1}^{\ell}\Re\left[(\beta_{jk}+\i\alpha_{jk})e^{\i k t}\right],
\end{equation}
where $\beta_{j0},\beta_{jk},\alpha_{jk}$, $j=1,2,\ldots,m$, $k=1,2,\ldots,\ell$, are $(1+2\ell)m$ real constants.
\end{thm}
\begin{proof}
The proof follows from Lemmas~\ref{L:S-Sk} and~\ref{L:Sk}.
\end{proof}

In view of~(\ref{eq:S-gam}), it is clear that the function $\psi$ in~(\ref{eq:psi}) is in the space $S^-$. By the uniqueness of the function $\psi$ in~(\ref{eq:RHp-2}) and the function $h$ in~(\ref{eq:RHp2}), we conclude that the functions $\psi$ and $h$ are identical, i.e., $\psi\equiv h$ where $h$ is given by~(\ref{eq:h}).

\section{Solving the general conjugation problem}

Since the function $\psi$ in~(\ref{eq:RHp-2}) and the function $h$ given by~(\ref{eq:h}) are identical, the RH-problem~(\ref{eq:RHp-w2}) or equivalently the general conjugation problem~\eqref{eq:RHp-w} can be solved by the integral equation with the generalized Neumann kernel~(\ref{eq:ie}) as in the following theorem.

\begin{thm}\label{T:RHpw-Psi}
For any $\gamma\in H$, the boundary values of the unique solution $\Psi$ of the general conjugation problem~\eqref{eq:RHp-w} are given by
\begin{equation}\label{eq:Psi-ie}
\Psi(\eta(t))=\frac{\gamma(t)+h(t)+\i\mu(t)}{A(t)}
\end{equation}
where $A(t)$ is defined by~\eqref{eq:A}, $\mu(t)$ is the unique solution of the integral equation~\eqref{eq:ie} and the function $h(t)$ is given by~\eqref{eq:h}.
\end{thm}

To evaluate the $(2\ell+1)m$ unknown real constants $a_{j0}$, $a_{jk}$ and $b_{jk}$  in~(\ref{eq:RHp-w}), $j=1,\ldots,m$, $k=1,\ldots,\ell$, we rewrite the function $\psi_{j}(t)$ in~(\ref{eq:psi}) as
\begin{equation}\label{eq:psi2}
\psi_{j}(t)=-\sum_{k=0}^{\ell}a_{jk}\cos kt+\sum_{k=1}^{\ell}b_{jk}\sin kt, \quad j=1,2,\ldots,m.
\end{equation}
By obtaining the real function $h$ from~(\ref{eq:h}) and since $\psi=h$, we have the following theorem.

\begin{thm}\label{T:RHpw-ajbj}
The values of the $(2\ell+1)m$ unknown real constants $a_{jk}$ and $b_{jk}$in~\eqref{eq:RHp-w} are given by
\[
a_{j0}=-\frac{1}{2\pi}\int_{0}^{2\pi}h_{j}(t)dt, \quad
a_{jk}=-\frac{1}{\pi}\int_{0}^{2\pi} h_{j}(t)\cos ktdt, \quad
b_{jk}= \frac{1}{\pi}\int_{0}^{2\pi} h_{j}(t)\sin ktdt,
\]
for $j=1,\ldots,m$ and $k=1,\ldots,\ell$ where the function $h(t)$ is given by~\eqref{eq:h}.
\end{thm}

In the above two theorems, we have used the integral equation~(\ref{eq:ie}) to develop a method for solving the RH problem~(\ref{eq:RHp-w}). We can also use Theorem~\ref{T:S-gam} and Corollary~\ref{C:g+h} to provide alternative proof for Theorem~\ref{T:Weg} for any $\gamma\in H$.

\begin{proof}[Alternative proof of Theorem~\ref{T:Weg}]
For any integer $\ell\ge0$, let the function $A$ be given by~\eqref{eq:A}. Then for any functions $\gamma\in H$, it follows from Corollary~\ref{C:g+h} that a unique function $h\in S^-$ exists such that the RH problem
\begin{equation}\label{eq:Af}
\Re[Af]=\gamma+h
\end{equation}
is uniquely solvable. The function $h$ is given by~(\ref{eq:h}) where $\mu$ is the unique solution of the integral equation~\eqref{eq:ie}. Then Theorem~\ref{T:S-gam} implies that unique values of the $(2\ell+1)m$ real constants $\beta_{j0},\beta_{jk},\alpha_{jk}$, $j=1,2,\ldots,m$, $k=1,2,\ldots,\ell$, exist such that
\[
h_{j}(t)=\beta_{j0}+\sum_{k=1}^{\ell}\Re\left[(\beta_{jk}+\i\alpha_{jk})e^{\i k t}\right].
\]
Hence the RH problem~\eqref{eq:Af} can be written as
\begin{equation}\label{eq:Af2}
\Re\left(A_{j}f_{|j}-\beta_{j0}-\sum_{k=1}^{\ell}\Re\left[(\beta_{jk}+\i\alpha_{jk})e^{\i k t}\right]\right)=\gamma_{j},
\end{equation}
which in view of the definition~\eqref{eq:A} of the function $A$ gives the proof of Theorem~\ref{T:Weg} where $\Psi=f$ and $a_{j0}=-\beta_{j0},a_{jk}=-\beta_{jk},b_{jk}=-\alpha_{jk}$, $j=1,2,\ldots,m$, $k=1,2,\ldots,\ell$.
\end{proof}

\section{Concluding remarks}

We have presented a boundary integral method for solving the general conjugation problem which is a certain class of Riemann-Hilbert problems. The method is based on a uniquely solvable boundary integral equation with the generalized Neumann kernel. We have also presented an alternative proof of Theorem~\ref{T:Weg} which has been proved previously in~\cite{Weg01r}. The numerical implementation of the above proposed method will be presented in future works.

%

\end{document}